\documentclass[a4paper,10pt,leqno]{article}
\usepackage[applemac]{inputenc}
\usepackage[T1]{fontenc}
\usepackage[english,francais]{babel}

\usepackage{geometry}
\usepackage{amsmath}
\usepackage{amssymb,amsmath,amsthm,amscd}
\usepackage{mathrsfs}

\usepackage{titlesec}
\titleformat{\section}[hang]{\normalfont\scshape}{\thesection.}{1em}{\centering}
\usepackage{enumerate}

\theoremstyle{plain} % style plain
\newtheorem{theo}{Théorème}[section]

\newtheorem{defi}{Définition}
\newtheorem{rema}{Remarque}

\newtheorem{prop}{Proposition}[section]

\newtheorem{hypo}{Hypothèse}

\numberwithin{equation}{section}
\title{UNE CONDITION SUFFISANTE POUR QUE LE PROBLÈME DE CAUCHY SOIT BIEN POSÉ}
\author{BERNARD LASCAR AND RICHARD LASCAR}
\date{}
\newcommand{\R}{\mathbb{R}}
\newcommand{\C}{\mathbb{C}}

\newcommand{\N}{\mathbb{N}}

\newcommand{\beeq}{\begin{equation}}
\newcommand{\eneq}{\end{equation}}
  
\begin{document}
\maketitle
\begin{center}
\it Dédié à la mémoire du Professeur Louis Boutet de Monvel
\end{center}
{\abstract On prouve ici une estimation d'énergie pour le problème de Cauchy pour des opérateurs hyperboliques à caractéristiques au plus doubles qui contient à la fois les cas non effectivement hyperboliques voir L. Hörmander [3] et les cas effectivement hyperboliques, voir R. Melrose [8].

We prove here an energy estimate for the Cauchy problem for hyperbolic equations with double characteristic, which contains both effectively and non effectively points, see L. Hörmander [3] and R. Melrose [8], in a unique framework.

}
%\section{ÉNONCÉ DU RÉSULTAT}
\section{Énoncé du Résultat}
Soit $P(t,x,D_t,D_x)$ un opérateur différentiel de degré $m$ sur $\R^{n+1}$ à coefficients $C^\infty$. On suppose que $p_m(t,x,\xi,\tau)$ est un polynôme hyperbolique, c'est à dire que l'équation $\tau\in\C\to p(x,\xi,\tau)$ a $m$ racines réelles pour tout $(t,x,\xi)$ réels. Cette conditions est bien sûr nécessaire. Ivrii-Petkov ont trouvé les conditions nécessaires pour que l'on puisse résoudre le problème de Cauchy [5]. L. Hörmander [2] a traité les cas non effectivement hyperboliques, R. Melrose [8] a résolu le cas effectivement hyperbolique. Il reste donc seulement à étudier des opérateurs près d'un point de transition entre ces deux cas. Voir les traités de référence [3] et [1].

On ne considère ici que des opérateurs à caractéristiques de multiplicité double et par des factorisations évidences, on se ramène au cas où $m=2$ et où $P(t,x,D_t,D_x)$ est un opérateur pseudo-différentiel de la forme :
\begin{align}\label{eq:1.1}
P &=-D_t^2+ A(t,x,D_x), A(t)\in C^{\infty}(\R,\Psi^2(\R^n)) ;\\
A'(t) & = \frac{1}2(A(t)+A^*(t))\in\Psi^2(\R^n) ;\notag\\
A''(t)&=\frac{1}{2i}(A(t)-A(t)^*)\in\Psi^1(\R^n),\notag\\
& \text{le symbole principal }a(t,x,\xi) \text{ de } A' \text{ est } \geq 0\notag
\end{align}
\begin{hypo}\label{hypo:1} On suppose qu'au voisinage de tout point $(0,X_0)\in\{0\}\times T^*\R^n$, il existe $C>0$ et $k\in\N$, $k\geq 1$ tels que :
\begin{enumerate}[i)]
\item 
\beeq\label{eq:1.2} 
a(t,X)=t^{2k}\alpha(t,X)+\beta(t,X),
\eneq
avec $\alpha(t,X)\geq 0$, $\beta(t,X)\geq 0$ et $|\partial_t\alpha(t,X)|\leq C\alpha(t,X)$, $|\partial_t\beta(t,X)|\leq C\beta(t,X)$.
\item La partie imaginaire du symbole sous-principal vérifie :
\begin{multline}\label{eq:1.3}
|a''_1|\leq C\Bigl(\frac{|\partial_t a|}{a^{1/2}}+a^{1/2}\Bigr)(t,X)=C(\mu+a^{1/2})(t,X),\\
\mu \text{ est la valeur propre réelle positive ou nulle de }F_p.
\end{multline}
\item Si $k\geq 2$ ou si $\alpha(0,X_0)=0$, c'est à dire si $(0,X_0)$ est non effectivement hyperbolique on suppose que :
\begin{enumerate}[a)]
\item $\Sigma_t=\beta_{t}^{-1}(0)\subseteq T^* \R^n$ est une variété $C^{\infty}$ de rang symplectique constant, sur laquelle $\beta_t$ s'annule exactement à l'ordre $2$, uniformément pour $t$ voisin de $0$.
\item 
\beeq\label{eq:1.4}
 a'_1(t,X)+\frac{1}{2}\text{tr}^+ F_a(t,X) >0 \text{ pour }(0,X)\in\Sigma, \text{ voisin de }(0,X_0).
\eneq
\end{enumerate}
\item Si $\alpha(0,X_0)>0$ et $k=1$ alors $(0,X_0)$ est effectivement hyperbolique, on ne suppose pas iii), tandis que (\ref{eq:1.3}) est automatiquement vérifiée.
\end{enumerate}
\end{hypo}
\begin{theo} Sous l'hypothèse (\ref{hypo:1}) le problème de Cauchy est bien posé dans $C^\infty$ sur $t=0$.
\end{theo}
\section{Preuve du Théorème}
On fait une estimation d'énergie pondérée.
\begin{prop} Il existe $C>0$ et $T>0$ petit, $s>0$ ne dépendant que de la taille de la partie imaginaire du symbole sous-principal tels que pour $u_1\in C^{\infty}([0,T]\times \R^n)$
\begin{multline}\label{eq:2.1}
\Lambda^{-1}\int_{0}^\infty \left(\|u_1\|^2 (s)\Bigl\|\Bigl(\frac{D_t}{\Lambda}\Bigr)u_1\Bigr\|(s)\right)^2\theta_0(s)ds\\
\leq C\left(\int_0^\infty \|Pu_1\|^2(t)\theta_0(t)dt+\Bigl(\|u_1(0)\|^2+\Bigl\|\Bigl(\frac{D_t}{\Lambda}\Bigr)u_1(0)\Bigr\|^2\Bigr)\Bigl(\int_0^\infty \theta_0(t)dt\Bigr)\right)\\
\text{pour }u_1\in C_0^{\infty}([0,T[\times \R^n).
\end{multline}
avec
\beeq\label{eq:2.2}
\theta(t,\Lambda)=\exp(s\ln (1+t\Lambda^{\tfrac{1}{k+1}})).
\eneq
\end{prop}
On quantifie avec un grand paramètre $\Lambda>0$, on note :
\beeq\label{eq:2.3}
(q^{w_\Lambda})u(x)=\int e^{i\Lambda(x-y)\xi} q\Bigl(\frac{x+y}{2},\xi,\Lambda\Bigr)u(y)\Bigl(\frac{\Lambda}{2\pi}\Bigr)^n dyd\xi.
\eneq
L'opérateur $A$ est classique, sa partie auto-adjointe s'écrit $A'=\sigma_1^{w_\Lambda}$ avec 
\begin{multline}\label{2.4}
\sigma_1(t,X)=a(t,X)+\Lambda^{-1}a_1(t,X,\Lambda)\\
\text{avec } a\geq 0, a\in S(1,\Gamma), a_1\in S(1,\Gamma) \text{ et }\Gamma=|dt|^2+|dX|^2.
\end{multline}
La partie non auto-adjointe de $A$ s'écrit $A''=\sigma_2^{w_\Lambda}$ où :
\beeq\label{2.5}
\sigma_2(t,X)=\Lambda^{-1}b_1(t,X), b_1=b_1^0 +b_2, b_1^0 \text{ vérifie (\ref{eq:1.3}) tandis que } b_2\in S(\Lambda^{-2},\Gamma).
\eneq
Pour prouver une borne inférieure pour $Q=q^{w_\Lambda}$ :
\beeq\label{eq:2.6}
q(t,\tau,X)=-(\tau-i\rho \Lambda^{-1})^2 +\sigma_1(t,X) + i\sigma_2(t,X),
\eneq
on calcul un multiplicateur $M=m^w$, $m=m'+i m''$ où $m'(t,X)=\theta(t,X)$, $m''=\tau n(t,X)$. En effet le calcul donne :
\begin{prop}\label{prop:2.2}
\beeq\label{eq:2.7}
\Re m'\# q' =\tau^2\theta +\tau\rho^2 \Lambda^{-2}+\frac{1}{4} \Lambda^{-2}\partial_t^2\theta +\Re\theta\#\sigma_1
\eneq

\beeq\label{eq:2.8}
\Re m''\# q'' =2\rho\tau^2\Lambda^{-1} n +\tau n\Lambda^{-1} b_1+\frac{1}{2} \Lambda^{-2}\Im n\#\partial_t\sigma_1
\eneq

\beeq\label{eq:2.9}
\Im m'\# q''=-\rho\Lambda^{-2}\partial_t\theta -\Lambda^{-1}\Im \theta\# b_1
\eneq

\beeq\label{eq:2.10}
\Im m''\# q'' =-\tau^2\Lambda^{-1}\partial_t n+\tau\Im\#\sigma_1- \frac{1}{2}\Lambda^{-1}\Re n\#\partial_t \sigma_1.
\eneq
Ce qui additionné donne :
\begin{multline}\label{eq:2.11}
\Re\overline{m}\# q=\tau^2(-\theta+2\rho\Lambda^{-1}n-\Lambda^{-1}\partial_t n)+\tau(\Lambda^{-1}nb_1+\Im n\#\sigma_1) +\\
\left(\theta\rho^2\Lambda^{-2}+\frac{1}{4}\Lambda^{-2}\partial_t^2 \theta +\Re\theta\#\sigma_1 +\frac{1}2\Lambda^{-2}\Im n\#\partial_tb_1-\right.\\
\left. \rho\Lambda^{-2}\partial_t\theta -\Lambda^{-1}\Im\theta\#b_1-\frac{1}{2}\Lambda^{-1}\Re n\#\partial_t\sigma_1\right)
\end{multline}
\end{prop}
On posera donc :
\beeq\label{eq:2.12}
c=-\theta+2\rho\Lambda^{-1}n-\Lambda^{-1}\partial_t n
\eneq

\beeq\label{eq:2.13}
e=\Lambda^{-1}nb_1+\Im n\#\sigma_1
\eneq

\begin{multline}\label{eq:2.14}
f=\left(\theta\rho^2\Lambda^{-2}+\frac{1}{4}\Lambda^{-2}\partial_t^2\theta +\Re\theta\#\sigma_1 +\frac{1}{2}\Lambda^{-2}\Im n\#\partial_t b_1-\right.\\
\left.\rho\Lambda^{-2}\partial_t\theta -\Lambda^{-1}\Im \theta\# b_1-\frac{1}2\Lambda^{-1}\Re n\#\partial_t\sigma_1\right).
\end{multline}
Donc (\ref{eq:2.11})  s'écrit :
\beeq\label{2.16}
r=c\tau^2+e\tau +f=c(\tau-\lambda)^2+\delta \text{ avec }\lambda=\frac{e}{2c} \text{ et }\delta=f-\frac{e^2}{4c}.
\eneq
Il faut donc réaliser $c>0$ et $\delta >0$.
\begin{proof}[Preuve de la proposition \ref{prop:2.2}]
Soit $Y$ le point courant des variables $(u,\rho,t,\tau,X)$, où $u$ est la variable duale de $\rho$. On utilise la formule :
\beeq\label{eq:2.17}
(a\# b)(y,\eta)\simeq \sum_{\alpha,\beta}\Bigl(\frac{1}{2 i\Lambda}\Bigr)^{|\alpha|+|\beta|}\frac{1}{\alpha!\beta!}(-1)^{|\beta|} \partial_y^\beta \partial_\eta^\alpha a(y,n)\partial_y^\alpha\partial\eta^\beta b(y,n).
\eneq
Tous les termes non explicites de (\ref{eq:2.11}) ne font intervenir que le calcul en $X$.

Pour minorer dans $L^2(\R_y^{n+2})$, $(\Re\overline m\# q)^{w_\Lambda}$ il faut introduire un calcul de Weyl et appliquer un résultat de bornes inférieurs.

Il va falloir utiliser une fonction de poids $\theta_0(t)$ pour tenir compte de la perte de dérivées due à la partie imaginaire du symbole sous-principal, on le prendra de la forme :
\begin{multline}\label{eq:2.18}
\theta_0(t)=\exp\Bigl(-s\int_0^t f_\mu^{-1}(\sigma)d\sigma\Bigr)\\
\text{avec } f_\mu=\mu^{-1}\langle t\mu\rangle \text{ où }\mu \text{ désigne un nouveau grand paramètre}.
\end{multline}
Il y a donc des conditions sur $\mu$ pour que le calcul symbolique au moins en $\tau,t$ soit admissible. On l'étudie donc ici.

En $X$ il n'y a pas de choix, $a\geq 0$ et $\sigma_1=a+\Lambda^{-1}a_1$, $a$ et $a_1$ sont dans $S(1,\Gamma)$. Soit $\mu\geq 1$ on pose :
\begin{defi}\label{defi:1} 
\beeq\label{2.19}
f_\mu(t)=(|t|+\mu^{-2})^{1/2}, \text{ comme }|t|\leq 1, \mu^{-1}\leq f_\mu(t)\leq 1+\mu^{-1}\leq 2.
\eneq
\end{defi}
Donc si $|t|\leq 1$, $|dt|\leq 2 f_\mu^{-1}|dt|$.
\begin{defi}\label{defi:2}
\beeq\label{eq:2.20}
G_Y=M^{-2}|du|^2+\langle\rho,\tau,\Lambda\rangle^{-2}|d\rho|^2 +f_\mu(t)^{-2}|dt|^2+\langle\rho,\tau,\Lambda\rangle^{-2}|d\tau|^2+|dX|^2.
\eneq
\end{defi}
La métrique duale de $G$ par rapport à $\Lambda^2\sigma$ est :
\begin{align}\label{eq:2.12}
G_Y^{\Lambda^2\sigma} &=\Lambda^2\Bigl[ \langle\rho,\tau,\Lambda\rangle^2|du|^2 +M^2|d\rho|^2 +\langle\rho,\tau,\Lambda\rangle^2|dt|^2+f_\mu(t)^2|d\tau|^2+|dX|^2\Bigr]\geq\\
&\geq \min\Bigl(M\Lambda\langle\rho,\tau,\Lambda\rangle,\Lambda\langle\rho,\tau,\Lambda\rangle f_\mu(t),\Lambda\Bigr)^2G_Y\geq \min(\Lambda^2\mu^{-1},\Lambda)^2 G_Y.\notag
\end{align}
\end{proof}
\begin{prop}\label{prop:2.3}
La métrique $g_{1,Y}=f_\mu(t)^{-2}|dt|^2+\langle\rho,\tau,\Lambda\rangle^{-2}|d\tau|^2$ est lent et $\Lambda^2\sigma$ tempérée si $1\leq \mu\leq \Lambda$. La fonction $h_1(Y)$ du calcul associé vaut :
\beeq\label{eq:2.22}
h_1(Y)=\Lambda^{-1}\langle\rho,\tau,\Lambda\rangle^{-1}\mu(\mu t)^{-1}\leq \mu\Lambda^{-2}.
\eneq
\end{prop}
\begin{proof}[Preuve de la proposition \ref{prop:2.3}]
La métrique $g_{1,Y}$ est lente car $f_\mu(t)^{-1}$ et que la fonction $x\in\R\to\langle x\rangle^s$ est lente pour la métrique $\langle x\rangle^{-2}|dx|^2$ pour tout $s\in\R$. En effet 
\begin{enumerate}[i)]
\item 
\beeq\label{eq:2.23}
C^{-1}\leq \frac{f_\mu(t)}{f_\mu(t_1)}=\frac{\langle t\mu\rangle}{\langle t_1\mu\rangle}\leq C\text{ si }|t_1-t|f_\mu(t)^{-1}\leq \varepsilon
.\eneq
\item 
\beeq\label{eq:2.24}
\frac{f_\mu(t)}{f_\mu(t_1)}\leq C(1+\mu|t-t_1|)\text{ pour tout }t \text{ et }t_1.
\eneq
\end{enumerate}
Comme 
\beeq\label{eq:2.25}
\frac{g_{1,Y_1}}{g_{1,Y}}\leq C(1+\mu|t-t_1|+|\tau-\tau_1|)^2\leq C\Bigl(1+g_{1,Y}^{\Lambda^2\sigma}\Bigr(t-t_1,\tau-\tau_1)) \text{ si }\mu\leq \Lambda.
\eneq
donc $g_1$ et $\Lambda^2\sigma$ tempérée si $1\leq \mu\leq \Lambda$.

Pour les mêmes raisons, la métrique $g_{0,Y}=M^{-2}|du|^2+\langle \rho,\tau,\Lambda\rangle^{-2}|d\rho|^2$  est lente et $\Lambda^2\sigma$ tempérée. On résume :
\end{proof}
\begin{prop}\label{prop:2.4} 
\begin{enumerate}[i)]
\item La métrique $G_Y$ est admissible pour le calcul de Weyl (voir [3]), avec une fonction $H(Y)=\max(\Lambda^{-2}\mu,\Lambda^{-1})=\Lambda^{-1}$, car $1\leq \mu\leq \Lambda$.
\item Si $|t|\leq 1$, $S(m,\Gamma)\subseteq S(m,G)$.
\item 
\beeq\label{eq:2.26} 
\Bigl|\Bigl(\frac{d}{dt}\Bigr)^jf_\mu(t)\Bigr|\leq Cf_\mu^{1-j},\text{ soit }f_\mu\in S(f_\mu,G).
\eneq
\end{enumerate}
\end{prop}
\begin{prop}\label{prop:2.5}
Soient $n_0\in\R$ et $s\in\R$ assez grands alors :
\beeq\label{eq:2.27}
\theta_0(t)=\exp\Bigl(-s\int_0^t f_\mu^{-1}(\sigma)d\sigma\Bigr)
\eneq
\beeq\label{eq:2.28}
n=n_1\exp\Bigl(-s\int_0^t f_\mu^{-1}(\sigma)d\sigma\Bigr)\text{ avec } n_1=n_0f_\mu(t)\Lambda.
\eneq
\begin{enumerate}[i)]
\item
\beeq\label{eq:2.29}
\left|\Bigl(\frac{d}{dt}\Bigr)^j \theta_0(t)\right|\leq \theta_0(t)Cf_\mu^{-j}
\eneq
\item
\beeq\label{eq:2.30}
\Bigl(\frac{d}{dt}\Bigr)\theta_0(t)=-sf_\mu^{-1}(t)\theta_0(t) \text{ et } \Bigl(\frac{d}{dt}\Bigr)^2\theta_0(t)=(s^2f_\mu^{-2}(t)+2st\mu^3\langle t\mu\rangle^{-3})\theta_0(t)>0.
\eneq
\item 
\beeq\label{eq:2.31}
-\partial_tn=\Lambda n_0(s-\mathcal{O}(1)))\geq cs\Lambda n_0 \text{ si } |s|\geq 2.
\eneq
Donc
\begin{multline}\label{eq:2.32}
c=-\theta_0+2\rho\Lambda^{-1}n- \Lambda^{-1}\partial_t n \geq \theta_0(c|s||n_0|-2\rho\mathcal{O}(|n_0|)f_\mu-1)\geq c\theta_0|n_0||s|\\
\text{si } sn_0\geq 2 \text{ et } \rho f_\mu\leq 1.
\end{multline}
\item 
\beeq\label{eq:2.33}
\partial_t^j \theta_0\simeq (-s)^j f_{\mu}^{-j}(t).
\eneq
\end{enumerate}
\end{prop}
Soit $\theta\in S(m_0,G)$ où $m_0$ est une fonction de poids $G$ admissible, on exprime les relations (\ref{eq:2.11}) :
\beeq\label{eq:2.34}
\Lambda^{-1}\partial_t n\in S(\theta_0,G),
\eneq
\beeq\label{eq:2.35}
-\Lambda^{-1}b_1 n\in S(f_\mu\theta_0,G),
\eneq
\beeq\label{eq:2.36}
\Im n\#\sigma_1 =\frac{1}2 \Lambda^{-1}\{n,\sigma_1\}+r_0, r_0\in S(\theta_0f_\mu \Lambda^{-2},G),
\eneq
\beeq\label{eq:2.37}
\theta\rho^2\Lambda^{-2}\in S(\theta_0\rho^2\Lambda^{-2},G)
\eneq
\beeq\label{eq:2.38}
\frac{1}4 \Lambda^{-2}\partial_t^2\theta \in S(\theta_0f_\mu^{-2}\Lambda^{-2},G)
\eneq
\beeq\label{eq:2.39}
\Re (\theta\#\sigma_1)=\theta\sigma_1+r_1, r_1\in S(\theta_0\Lambda^{-2},G),
\eneq
\beeq\label{eq:2.40}
\frac{1}{2} \Lambda^{-2}\Im n\#\partial_t b_1 \in S(m_0\Lambda^{-2},G),
\eneq
\beeq\label{eq:2.41}
-\rho\Lambda^{-2}\partial_t\theta \in S(\theta_0\rho \Lambda^{-2} f_\mu^{-1},G),
\eneq
\beeq\label{eq:2.42}
-\Lambda^{-1}\Im (\theta\#b_1)\in S(\theta_0\Lambda^{-2} f_\mu^{-1},G),
\eneq
\beeq\label{eq:2.43}
-\frac{1}{2}\Lambda^{-1}\Re n\#\partial_t\sigma_1 =-\frac{1}{2}\Lambda^{-1}n\partial_t\sigma_1+r_2, r_2\in S(\theta_0\Lambda^{-2}f_\mu,G).
\eneq
On prendra si $\theta(t,X)=\theta_0(t)\theta'(X),\theta'(X)\in S(1,\Gamma)$.

Le symbole $\Re \overline m\# q$ est un polynôme en $\tau$ de degré $2$, que l'on écrit :
\beeq\label{eq:2.44}
\Re \overline m\# q=c\tau^2+2 e\tau +f.
\eneq
\begin{prop}\label{prop:2.6}
\beeq\label{eq:2.45}
c\in S(\theta_0(1+\rho f_\mu),G), e\in S(\theta_0 f_\mu,G) \text{ et } f\in S(\theta_0,G)
\eneq
ce qui s'écrit 
\beeq\label{eq:2.46}
c(\tau-\lambda)^2+f-\frac{e^2}{4c}, \text{ avec }\lambda=\frac{e}{2c}\in S(\rho,G), \delta=f-\frac{e^2}{4c}\in S(\theta_0,G).
\eneq
\end{prop}
avec
\beeq\label{eq:2.47}
c=\theta_0(-1+2n_0\rho f_\mu + n_0 s)\geq cn_0(s+s\rho f_\mu)\theta_0\text{ dans } s\text{ grand et }n_0s\text{ grand}. 
\eneq
\beeq\label{eq:2.48}
e=\Lambda^{-1}n\theta b_1^0-\frac{1}{2}f_\mu n_0\theta_0\{\theta',\sigma_1\}+r_0,r_0\in S(\theta_0\Lambda^{-2}f_\mu,G).
\eneq
On majore $\frac{|e|^2}{4c}$ à l'aide des équations \ref{eq:2.48} et \ref{eq:1.3} soit :
\beeq\label{eq:2.49}
|e|\leq Cn_0f_\mu\theta_0\Bigl(\frac{|\partial_t a|}{a^{1/2}}+a^{1/2}+\Lambda^{-1}\Bigr).
\eneq
Donc 
\beeq\label{eq:2.50}
\frac{|e|^2}{4c}\leq Cn_0 \theta_0f_\mu^2(s+\rho f_\mu)^{-1}\left(\Bigl(\frac{|\partial_t a|}{a^{1/2}}\Bigr)^2+\Lambda^{-2}\right)
\eneq
Le terme dominant de $f$ est 
\beeq\label{eq:2.51}
f_0=-\frac{1}{2}\Lambda^{-1}n\partial_t\sigma_1=-\frac{1}{2}n_0f_\mu(\partial_t a+\mathcal O(\Lambda^{-1})).
\eneq
Nos hypothèses font que 
\beeq\label{eq:2.52}
\frac{|\partial_ta|}{a^{1/2}}\leq C t^{-1}a^{1/2}.
\eneq
\begin{enumerate}[i)]
\item Dans $I=\{t ; |t|\geq \mu^{-1}\}$ on majore $f_0$ par 
\beeq\label{eq:2.53}
|f_0|\leq Cn_0(a+n_0 t f_\mu\Lambda^{-1}).
\eneq
et
\beeq\label{eq:2.54}
\frac{|e|^2}{4c}\leq Cn_0s^{-1}(a+\Lambda^{-2}).
\eneq
\beeq\label{eq:2.55}
\delta\geq \theta_0 (a(1-Cn_0)+\Lambda^{-1}(a_1-Cn_0)-C\Lambda^{-2}).
\eneq
\item Dans $J=\{t;|t|\leq \mu^{-1}\}$ on majore $f_0$ par
\beeq\label{eq:2.56}
|f_0|\leq Cn_0(f_\mu(|\partial_t a|+\Lambda^{-1}))\leq Cn_0f_\mu (t^{2k-1}|\alpha|+\beta +\Lambda^{-1})\leq Cn_0(\mu^{-2k}\alpha+a+\Lambda^{-1}).
\eneq
Ce qui s'absorbe dans $\mu^2\Lambda^{-2}$ si $\mu\geq \Lambda^{\frac{1}{k+1}}$. Si $\mu^{-1}\leq \Lambda^{-1}$ on a donc :
\end{enumerate}
tandis que
\beeq\label{eq:2.57}
\frac{|e|^2}{4c}\leq Cn_0 s^{-1}\mu^{-2}(t^{k-1}\alpha^{1/2}+\beta^{1/2}+a^{1/2})^2 +\Lambda^{-1}\leq Cn_0(\mu^{-2k}\alpha+a+\Lambda^{-1}).
\eneq
Donc si $\mu\geq \Lambda^{\frac{1}{k+1}}$
\beeq\label{eq:2.58}
\delta\geq \theta_0(a+\Lambda^{-1}a_1-Cn_0\Lambda^{-1}-C\Lambda^{-2}).
\eneq
Nos hypothèses entraînent quand $\alpha(0,X_0)=0$ ou $k\geq 2$, que l'on peut appliquer l'inégalité de Melin voir [3] à $\beta+\Lambda^{-1}\alpha'_1$.
\begin{prop}\label{prop:2.7} Sous les hypothèses (\ref{hypo:1}) si $s$, et $\Lambda^{-1}\mu^2$, et $\Lambda\mu^{-1}$ sont grands
\beeq\label{eq:2.59}
\delta^{w_\Lambda}\geq c\Lambda^{-1}.
\eneq
\end{prop}
On minore $(c(\tau-\lambda))^2$ par calcul symbolique et
\beeq\label{eq:2.60}
\Bigl(c(t)(\tau-\lambda)^{w_\Lambda}\Bigr)=(\tau-\lambda\#)c(t)\#(\tau-\lambda)-\frac{1}{4}c''(t)\Lambda^{-2}+S(\Lambda^{-2},G).
\eneq
Comme $c(t)=-\theta+2n_0\rho f_\mu-n_0 f_\mu\partial_t\theta_0$,
$$c''(t)=-\partial_t^2\theta_0+\mathcal O(n_0)s^3 f_\mu^{-2}.$$
Comme $c(t)=\Lambda^{\frac{1}{k+1}}$, $f_\mu^{-2}\Lambda^{-2}=o(1)\Lambda^{-1}$ si $k\geq 2$.

Par ailleurs, les estimations se microlocalisent effectivement auprès de chaque point $(0,X_v)$, par une partition de l'unité sous la forme :
\beeq\label{eq:2.61}
\sum_v \chi_v^2(t,X)=1,\chi_v\in S(1,G),
\eneq
seuls vont perturber l'inégalité les bi-commutateurs avec les $\chi_v$ qui produisent $\mathcal O(f_\mu^{-2}\Lambda^{-2})$.
\begin{multline}\label{eq:2.62}
\int_t^\sigma \left(\delta+\Bigl(c(\tau-\lambda)^2\Bigr)^{w_\Lambda}u_1,u_1(s)ds\right)e^{-2\rho s}\geq c\Lambda^{-1}\left(\int_t^\sigma \left|\Bigl(\frac{D_t}{\Lambda}-\lambda\Bigr)u_1(s)\right|^2+|u_1|^2(s)e^{-2\rho s}ds\right)\\
-C\Lambda^{-1}\theta_0(t)\Bigl(|u_1|^2(t)(t)+\Bigl|\Bigl(\frac{D_t}{\Lambda}-\lambda\Bigr)u_1(t)\Bigr|^2\Bigr)e^{-2\rho t}\\
- C\Lambda^{-1}\theta_0(\sigma)\Bigl(|u_1|^2(\sigma)e^{-2\rho t}+\Bigl|\Bigl(\frac{D_t}{\Lambda}-\lambda\Bigr)u_1(\sigma)\Bigr|^2\Bigr)e^{-2\rho t}.
\end{multline}
Puis on voit que :
\beeq\label{eq:2.63}
C\left|\Bigl(\frac{D_t}{\Lambda}-\lambda\Bigr)u_1\right|^2\geq \left|\frac{D_t}{\lambda}u_1\right|^2 -C\|\lambda\|^2|u_1|^2.
\eneq
Puis
\beeq\label{eq:2.64}
\|\lambda\|\leq C.
\eneq
On déduit de (\ref{eq:2.63}) et de (\ref{eq:2.62}) que :
\beeq\label{eq:2.65}
\int_t^\sigma \varphi(s)ds\leq C\int_t^\sigma\left(\Bigl(\delta+\Bigl(c(\tau-\lambda)^2\Bigr)^{w_\Lambda}\Bigr)u_1,u_1(s) ds\right)e^{-2\rho s}+C(\varphi(t)+\varepsilon\varphi(\sigma)).
\eneq
où
\beeq\label{eq:2.66}
\varphi(s)=\theta_0(s) e^{-2\rho s}\Lambda^{-1}\left(|u_1(s)|^2+\Bigl|\Bigl(\frac{D_t}{\Lambda}-\lambda\Bigr)u_1\Bigr|^2(s)\right).
\eneq
On multiplie par $e^{-k\sigma}$ et on intègre de $0$ à $\infty$, on obtient quand $\varphi(0)=0$ :
\begin{multline}\label{eq:2.67}
\frac{1}{k}\int_0^\infty \varphi(s)e^{-ks}ds \leq C\int_0^\infty\left(\Bigl(\delta+\Bigl(c(\tau-\lambda)^2\Bigr)^{w_\Lambda}u_1,u_1(s)ds\Bigr)\right)\theta_0 e^{-(2\rho+k)s}ds +\\
C\int_0^\infty \varphi(s)e^{-ks}ds.
\end{multline}
On obtient ainsi :
\begin{multline}\label{eq:2.68}
\left(\Lambda^{-1}\int_0^{\infty}\Bigl(|u_1|^2(s)+\Bigl|\frac{D_t}{\Lambda}u_1\Bigr|(s)^2\Bigr)\theta_0(s)ds\right)\leq C\int_0^\infty |Pu_1|^2(t)\theta_0(t)dt\\
\text{pour } u_1\in C_0^\infty([0,T[\times\R^n).
\end{multline}
On peut ensuite modifier la condition de support en rajoutant à (\ref{eq:2.68}) les termes de bord en $0$ soit 
\begin{multline}\label{eq:2.69}
\Lambda^{-1}\int_0\infty \left(|u_1|^2(s)+\Bigl|\frac{D_t}{\Lambda} u_1\Bigr|(s)^2\right)\theta_0(s)ds\leq\\
C\left(\int_0^\infty |Pu_1|^2(t)\theta_0(t)dt + \Bigl(|u_1(0)|^2+\Bigl|\frac{D_t}{\Lambda}u_1(0)\Bigr|^2\Bigr)\Bigl(\int_0^\infty \theta_0(t)dt\Bigr)\right)\text{ pour }u_1\in C_0^\infty([0,T[\times \R^n).
\end{multline}
Ce qui achève la preuve.
\begin{rema}\label{rema:1} Quand $k=1$ et $\alpha(0,X_0)=0$, on choisit $\mu-\mu_0\Lambda^{1/2}$, il faut alors tout d'abord ajuster le signe de $n_0$ de sorte que
\beeq\label{eq:2.70}
f_0=\frac{1}{2}n_0\partial_t a+\mathcal O(f_\mu\Lambda^{-1})\geq -C|n_0|a+\mathcal O(f_\mu\Lambda^{-1}).
\eneq
Il faut ensuite ajuster tous les paramètres, on avait :
\begin{enumerate}[a)]
\item $|n_0|\leq \varepsilon$ et $n_0\leq 0$ et $s<0$,
\item $|s|^{-1}|n_0|\mu_0^{-2}\alpha\leq \varepsilon$,
\item $|s|^3|n_0|\mu_0^2\leq \varepsilon$,
\item $|s|\geq 1$,
\item $|s||n_0|\geq 1$.
\end{enumerate}
Qui sont satisfaites si $\alpha(0,X_0)\leq \varepsilon^2$ avec $\sigma=s|n_0|^{-1}=1$, $\tau=(n_0^{-1}\mu_0)^2=\varepsilon$, puis $n_0$ petit et $s$ grand.
\end{rema}

\begin{rema}\label{rema:2} On peut aussi traiter le cas effectivement hyperbolique de la façon suivante. On prend comme toujours :
\beeq\label{eq:2.71}
n=-|n_0|\Lambda f_\mu\theta_0
\text{ avec }n_0<0\text{ petit},
\eneq
Et
\beeq\label{eq:2.72}
\mu=\mu_0\Lambda^{1/2}
\eneq
Et on pose :
\begin{multline}\label{eq:2.73}
c=\widetilde c\theta_0=\theta_0(-1+2\rho f_\mu n_0 +sn_0 +\mathcal O(s^{-1})),\\
\text{on pose }v=-1+sn_0(1+\mathcal O(s^{-1})) \text{ et on prend }s<0, |s|\text{ grand}.
\end{multline}
On veut maintenant :
\beeq\label{eq:2.74}
v>0\text{ petit},
\eneq
Le défaut du calcul symbolique (\ref{eq:2.60}) doit être compensé par le terme $\frac{1}4\Lambda^{-2}\partial_t^2\theta_0\simeq s^2 f_\mu^{-2}\Lambda^{-2}$, le terme $\frac{1}{4}\Lambda^{-2}c''(t)=\mathcal O(v)s^2 f_\mu^{-2}/\Lambda^{-2}$, il faut donc $v\leq \varepsilon$, où en fait $\varepsilon$ est une constante universelle.
\begin{enumerate}[i)]
\item La zone $\{t,|t|\geq \mu^{-1}\}$ se traite à l'aide de :
\beeq\label{eq:2.75}
\frac{|e|^2}{4c}\leq C|n_0|^2 t^2 \theta_0 v^{-1}
\eneq
Tandis que :
\beeq\label{eq:2.76}
f_0\geq \theta_0(\alpha|n_0|t+(1-C|n_0|)a+\mathcal O(\Lambda^{-1}))
\eneq
ce qui produit donc les conditions : $\Lambda^{-1}\leq C|t|(|n_0|+t)$, et $|t||n_0|\leq v$ ; qui seront satisfaites si $\mu_0\leq C^{-1}$ et $n_0\leq C^{-1}$. On produit donc ici une borne inférieure adéquate.
\item La zone $\{t,|t|\leq \mu^{-1}\}$

On utilise cette fois :
\beeq\label{eq:2.77}
\frac{|e|^2}{4c}\leq C|n_0|\mu^{-2}\theta_0 v^{-1} =\mathcal O(\mu_0^{-2}\Lambda^{-1}|n_0|v^{-1}\theta_0)
\eneq
Tandis que $s^2 f_\mu^{-2}\Lambda^{-2}$ produit lui $s^2\mu_0^2\Lambda^{-1}$. On prendra donc $|s|$ grand.
\end{enumerate}
\end{rema}
CQFD.

\textsc{Bernard Lascar. Richard Lascar. Université Denis Diderot. Département De Mathématiques.
Institut Mathématiques De Jussieu, UMR7586, 4 Place Jussieu, 75005 Paris. France}

{\it E-mail address :} {\tt richard.lascar@imj-prg.fr}
\end{document}